\documentclass[10pt]{article} 

\usepackage[utf8]{inputenc} 

\usepackage{geometry} 
\usepackage{graphicx} 

\usepackage{CJK}

\usepackage[colorlinks]{hyperref}
\usepackage{ulem}
\usepackage[T1]{fontenc}
\usepackage{calligra}

\usepackage{graphicx}
\usepackage{dcolumn}
\usepackage{bm}

\usepackage{times}
\usepackage{verbatim}
\usepackage{graphicx}
\usepackage{graphics,epsfig}

\usepackage{theorem}
\usepackage{makeidx}
\usepackage{amsmath}
\usepackage{epic}
\usepackage{amscd}
\usepackage[dvips]{xcolor}
\usepackage{siunitx}
\usepackage{lipsum}
\usepackage{mathrsfs}
\usepackage{bm}

\usepackage{bbm}
\usepackage{xy}
\usepackage{amssymb}
\usepackage{xcolor}
\usepackage{cancel}

\usepackage{bm}
\usepackage{booktabs} 
\usepackage{array} 
\usepackage{paralist} 
\usepackage{verbatim} 
\usepackage{subfig} 

\usepackage{fancyhdr} 
\usepackage{sectsty}
\allsectionsfont{\sffamily\mdseries\upshape} 

\usepackage[nottoc,notlof,notlot]{tocbibind} 
\usepackage[titles,subfigure]{tocloft} 

\usepackage{tikz}
\usetikzlibrary{arrows}
\tikzstyle{block}=[draw opacity=0.7,line width=1.4cm]

\newtheorem{theorem}{Theorem}
\newtheorem{lemma}{Lemma}
\newtheorem{definition}{Definition}

\newtheorem{proposition}{Proposition}

\newtheorem{remark}{Remark}
\newtheorem{corollary}{Corollary}

\DeclareFontFamily{OT1}{pzc}{}
\DeclareFontShape{OT1}{pzc}{m}{it}{<-> s * [1.2200] pzcmi7t}{}
\DeclareMathAlphabet{\mathpzc}{OT1}{pzc}{m}{it}

\title{Fast rotating flows in high spatial dimensions}
\author{Zhu, Jian-Zhou @ SCCFIS.org}

\begin{document}
\maketitle
\abstract{The central result about fast rotating-flow structures is the Taylor-Proudman theorem (TPT) which connects various aspects of the dynamics. Taylor's geometrical proof of TPT is reproduced and extended substantially, with Lie's theory for general frozen-in laws and the consequent generalized invariant circulation theorems, to compressible flows and to $d$-dimensional Euclidean space ($\mathbb{E}^{d}$) with $d\ge 3$. The TPT relatives, the reduced models (with particular interests on passive-scalar problems), the inertial (resonant) waves and the higher-order corrections, are discussed coherently for a comprehensive bird view of rotating flows in high spatial dimensions.}

\section{Introduction}

The vortex dynamics in models of classical and super flows, and, superconductors have long been studied and, recently, singular structures in spatial dimension $d\ge 4$ attract interests\cite{LinCPAM1998,JianLiuActaMathematicaSinica2006,MiyazakiPhD2010,ShashikanthJMP12,KhesinMMJ12}. Key roles in such vortex structures are played by the (fast) rotations, and central to the various aspects of the dynamics of the rotating structures is the celebrated Taylor-Proudman theorem (TPT).

Taylor's\cite{Taylor} analytical proof of TPT is based on invariant circulation theorem, thus geometrical and topological. Such a legacy was inherited, to our best knowledge, only by Chandrasekhar\cite{Chandrasekhar61Book} and only half-way in his own style, in the rotating frame [with a Coriolis term whose influence on the (regularity of the) solutions has recently been of intensive studies in mathematical literatures\cite{BabinMahalovNicolaenkoAsymptotAnal1997,KoniecznyYonedaJDE2011,SunYangCuiAMPA2016,
ZhaoWangMMAS,KishimotoYonedaMathAnn2017,KozonoMashikoTakadaJEE2014}]. Chandrasekhar did not use the Alfv\'en theorem to derive the analogue of TPT in incompressible magnetohydrodynamics: Taylor's geometrical idea then appears to be basically ignored again (but see, e.g., Roberts \cite{Roberts1978LNP} for some occasional notes). Indeed, Taylor's proof seems a detour, compared to the popular text-book proof checking directly the dominant term in the vorticity equation (essentially the simple idea of dominant balance). It is a mystery to us why Taylor did it that way and how he found it, and, interestingly, the closely related \textit{Taylor column} had already been observed (c.f., Velasco Fuentes\cite{VelascoFuentesEJMB2009} for more historical discussions and references) by the same scholar whose name (then William Thompson) is associated to the circulation invariance, the `Kelvin circulation theorem'\footnote{The earlier discovery of this theorem by Hankel has been addressed by Truesdell\cite{Truesdell1954Book} and recently by Frisch \& Villone\cite{FrischVilloneEPJH2014}.}, that he applied but did not cite as now people commonly would do. We will not explore this problem of history of science in this note but will show that Taylor's insight is fundamental and universal, when extended to compressible models and to higher dimensions using Lie's theory\cite{Arnold1989Book}, and, useful' for theoretical reasoning and physical applications.

Let's start from Taylor's own proof:--- 

The Stokes lemma\cite{Arnold1989Book}, transforming the surface integral to the integration around the circuit bounding it, connects the Helmholtz theorem of the Euler equation for an inviscid flow to Kelvin's circulation theorem. If the pressure term does not contribute (as in the barotropic case), for any material circuit $c(t)$ with horizontal/$x$-$y$ plane projection area $\mathcal{A}(t) = \oint_{c(t)} xdy - ydx$, we have the following invariant circulation for the flow of velocity $\bm{u}$ of a fluid with background rotation at a rate $\bm{\Omega}$ in the vertical direction ($z$ axis), relative to which is the velocity $\bm{u}' = \bm{u} - \bm{\Omega} \times \bm{r}$:
\begin{equation}
 \oint_{c(t)}\bm{u} \cdot d\bm{r} = \oint_{c(t)} (u_x' - y\Omega) dx + (u_y' + x\Omega) dy + u_z' dz = \oint_{c(t)}\bm{u}' \cdot d\bm{r} + 2\Omega \mathcal{A}. \label{eq:circulation}
\end{equation}
If $\oint_{c}\bm{u}' \cdot d\bm{r}$ for the relative motion velocity $\bm{u}'$ varies slowly (compared to $\Omega$), $\mathcal{A}$ changes little, which is the geometrical argument of Taylor \cite{Taylor} for the two-dimensionalization of incompressible rotating flows:
\begin{equation}
\text{For $\nabla \cdot \bm{u} =0$ and $\Omega \to \infty$, $d\mathcal{A}/dt\to 0 \Rightarrow \partial_z\bm{u} \to \bm{0}$}.
\end{equation}

We see that adding the gradient of a potential $\nabla \theta$ to $\bm{u}'$ formally does not change the result. In other words, we can consider the compressible flows, in which case TPT may be even more `powerful' in the sense that some `rotation effect' may reveal. Thus, the asymptotic (in the sense of `small' and `slow' limit for the vortical part of the relative motion) TPT of compressible rotating flows reads
\begin{eqnarray}
\text{$\partial_z \bm{u}_h = (\partial_z u_x, \partial_z u_y, 0) \to \bm{0}$ and the horizontal incompressibility $\nabla_h \cdot \bm{u}_h :=\partial_x u_x + \partial_y u_y \to 0$,} \label{eq:cTPE}\\
\text{but with no constraint on $\partial_z u_z$.} \label{eq:cTPEuz}
\end{eqnarray}

Taylor's proof reveal the time-independent version of Kelvin's circulation theorem, which actually is part of the flow geometry with, for instance, unsteady and invariant \textit{integral surfaces}\cite{Lee2000Book,Montgomery2002Book,FeckoJGP17} generalizing the vortex lines. 
The language of differential forms appears to expose the geometry and even the topology of flows more straightforwardly, particularly in higher dimensions and curved spaces, and if indeed Taylor's geometrical argument is essential, it deserves to check similar consequences of a background fields relevant to various Helmholtz-Kelvin-type theorems in different systems\cite{FeckoJGP17,ArnoldKhesin98Book,BesseFrischJFM17}. Note that a system may have a hierarchy of geometrical objects, besides the vorticity, say, for the classical Euler equation, satisfying the Helmholtz-Kelvin-type theorems, which makes it worth checking whether there would be accordingly a hierarchy of corrections or even inconsistency.

One of our motivations came from the passive-scalar problem: As shown in the above, the classical TPT for fast rotating incompressible fluids results in a  `cylinder condition' with everything independent of one of the vertical/rotating axis, which presents in the reduced dynamical model\cite{BabinMahalovNikolaenkoEJMB1996,EmbidMajdaCPDE1996}. It is direct to see that in higher-dimensional space\footnote{\label{lb:dDflow}High-dimensional flows have been of long theoretical interests, and have recently received intensive attendions on the issues of regularity\cite{DongDuCMP2006,DongGuDPDE2014,GuJDE2014,LiuWangJDE2017}, conservation laws \cite{KhesinChekanovPhD1989,ArnoldKhesin98Book}, specific dynamics in four space-time\cite{SubinVedanZAMP2004} or spatial\cite{ShashikanthJMP12} dimensions, and turbulence\cite{KraichnanPoF1985,MeneveauNelkinPRA1989,GotohPRE2007,NikitinJFM2011}.}, the Euler or Navier-Stokes equation with the cylinder condition also contains (multiple) passive scalar(s). Thus, it is intriguing whether the same physical mechanism in higher-dimensional space can be responsible for the cylinder condition, thus the passive scalar and the reduced model. A specific problem is the following. Suppose we have two passive scalars advected by a Navier-Stokes fluid in the three-dimensional (3D) Euclidean space $\mathbb{E}^3$, which can be viewed as the reduction from the five-dimensional flow with cylinder conditons on two of the coordinates. If we want to have some unity in the treatment to explore the properties of the scalars, a naive idea can be that we might obtain new perspectives by simply setting up fast rotations in $\mathbb{E}^5$ from which some generalized form of TPT would imply the necessary cylinder conditions. The definite answer to such a problem, its generalizations and the relative issues requires a systematic study of flows with fast rotation(s) in high spatial dimensions.

For a neat presentation, most of the main points of our results will be summarized in the  \textit{lemma}-\textit{theorem}-\textit{corollary} fashion, which should not be understood to be the signal of requiring strong mathematical background or high degree of mathematical rigor from the audiences. Especially the writing should be well understandable for fluid mechanists who, if not, are however suggested to get acquainted with the language of \textit{differential forms}, from the works of, say, Arnold\cite{Arnold1989Book}, Besse \& Frisch\cite{BesseFrischJFM17} and Tur \& Yanovsky\cite{TurYanovskyJFM93}, containing friendly written necessary materials in a nutshell, or from standard texts for more comprehensive knowledge.

\section{Geometrical remarks and models}\label{sec:geometricalRemarks}
In the following first subsection, we offer the relevant fluid-geometry remarks, mainly just to warm up with the notions and backgrounds for the lemma in the second subsection and later developments.

\subsection{Flows in the inertial frame: integral surface, (high-order) Kelvin theorems}

Let $\verb"U" := \sum_{i=1}^d u_i \verb"d" x_i$ be an element of 1-forms in the space $\mathfrak{g}^*$ dual to the Lie algebra $\mathfrak{g}$ of the incompressible velocity vector $u=\sum_i u_i\partial_i$ (or $\bm{u}$ for notation of a field), then its exterior derivative $\verb"d" \verb"U"$ is the vorticity 2-form.
With the introduction of the Bernoulli function $B=P+\bm{u}^2/2$ (we are not interested in the so-called conservative force as the gradient of a potential which is usually also included), the Lie derivative $L_{u}$ along $\bm{u}$ and the interior product $\iota_u$ by $u$, the ideal Euler equation for the dynamics of such objects in $\mathbb{E}^d$ reads [for reference of correspondence with the probably more familiar equations in the vector form, c.f., Eqs. (\ref{eq:continuity},\ref{eq:uh3dsr},\ref{eq:uv3dsr},\ref{eq:uh4dsr},\ref{eq:uv4dsr}) below]
\begin{eqnarray}
(\partial_t + \iota_u \verb"d")\verb"U" =(\partial_t + L_{u}) \verb"U" - \verb"d"\iota_u \verb"U"=(\partial_t + L_{u}) \verb"U" - \verb"d"\bm{u}^2=-\verb"d" B,\label{eq:iHDu}\\
(\partial_t + L_{u})\verb"d" \verb"U" = 0.\label{eq:iHDo}
\end{eqnarray}

Since $\verb"d" \verb"U"$ is Lie invariant, the \textit{exact} 4-form
\begin{equation}\label{eq:bivorticity}
\verb"d" \verb"U" \wedge \verb"d" \verb"U" = \verb"d"(\verb"U" \wedge \verb"d" \verb"U")
\end{equation}
is also invariant, and its Hodge dual $\star(\verb"d" \verb"U" \wedge \verb"d" \verb"U")$ is 
a 1-form in $\mathbb{E}^5$. 
The vanishing interior product $\iota_{\dot{\gamma}} \verb"d"(\verb"U" \wedge \verb"d" \verb"U") = 0$ defines a distribution $\mathcal{D}$\cite{Lee2000Book,Fecko06Book} whose integral manifold is tracked by the flow $\gamma(t)$. 
The integrability\footnote{Trivial for such one-dimensional distribution though, it is nice to go through the formal mathematical procedure in this beginning of gemoetrical remark, in case of necessity of guidance for readers not so familiar with the background.} of $\mathcal{D}$ is assured by the Frobenius theorem (involutivity) which can be validated directly: Suppose $\dot{\gamma}_1$, $\dot{\gamma}_2 \in \mathcal{D}$, we have 
\begin{equation}
\iota_{[\dot{\gamma}_1,\dot{\gamma}_2]} = [L_{\dot{\gamma}_1},\iota_{\dot{\gamma}_2}]= [\iota_{\dot{\gamma}_1}\verb"d" + \verb"d" \iota_{\dot{\gamma}_1}, \iota_{\dot{\gamma}_2}],\   \text{thus} \ \iota_{[\dot{\gamma}_1,\dot{\gamma}_2]} \verb"d" (\verb"U" \wedge \verb"d" \verb"U") = 0.
\end{equation}
So, for $d=5$, the integral manifold is precisely a line (the `bi-vorticity line', if a terminology is needed). 
We then have the Helmholtz-type theorem for the frozen-in bi-vorticity line. 
This result is obviously extendable to any odd dimension $d=2m+1$, by working with $(\verb"d" \verb"U")^{2m}$ in a similar way, i.e.,
\begin{equation}
\text{the distribution $\mathcal{D}$ defined by} \ \iota_{\dot{\gamma}} (\verb"d" \verb"U")^{2m}. \label{eq:2mD}
\end{equation}

Note that the interior product with the volume form $\iota_{\bm{\omega}} \mu = \verb"d" \verb"U" \wedge \verb"d" \verb"U"$ defines the vorticity vector $\bm{\omega}$ in $\mathbb{E}^5$, which also generates a `vorticity line'\cite{ArnoldKhesin98Book} which however should not be confused with the above bi-vorticity line $\gamma(t) \leftrightarrow \bm{r}(t)$. 

The lines can of course be generalized to \textit{integral surfaces} of the distributions of dimension $\Delta<d-(2\delta)$ (because the object is not decomposable in general: c.f., Appendix C of Fecko\cite{FeckoJGP17})  with $\delta < m$ replacing $m$ in $(\verb"d" \verb"U")^{2m}$, to have $(\verb"d" \verb"U")^{2\delta}$, in Eq. (\ref{eq:2mD}) for defining the distribution $\mathcal{D}$.

What is of closer relevance to our work when carried over to rotating frames is the \textit{integral surface}, with dimension $\Delta \le d-2$, of the \textit{distribution} in $\mathbb{E}^d$
\begin{equation}\label{eq:dUdistribution}
  \mathcal{D}:=\{ \ \text{vectors $\dot{\gamma}$ such that} \ \iota_{\dot{\gamma}}\verb"d" \verb"U" = 0 \ \text{holds} \ \}.
\end{equation}
When $d=3$, the integral surface is the vorticity line which is of traditional importance.\cite{WuMaZhou} With $\xi = \partial_t + u$ and $\Psi = \verb"U" - B\verb"d" t$ in the extended phase space spanned by both the spatial and temporal components (studying hydrodynamics in the space-time Euclidean, not Minkovskii, space has also a long history: c.f., Subin \& Vedan\cite{SubinVedanZAMP2004} and references therein), and, with $\diamond$ 
adding to $\verb"d"$ the differential with respect to $t$, the Euler equation is the interior product between $\xi$ and $\diamond \Psi$ 
\begin{eqnarray}
\iota_{\xi} (\diamond 
\Psi) = 0,\label{eq:succinct}
\end{eqnarray}
which is used to derive the local (high-order) invariants\cite{SubinVedanZAMP2004} and to introduce over the flowing closed chain $c(t)$ the \textit{relative invariant} $\int_c  \verb"U"$ which includes the well-known Kelvin theorem of circulation invariance (with respect to the flow $u$)\cite{FeckoJGP17}. Of course there are also other higher-order dynamical integral invariants corresponding to the above mentioned higher-order Lie-invariant forms.

Actually, right wedge product with $\verb"d"\verb"U"$ of Eq. (\ref{eq:iHDu}) and left wedge product with $\verb"U"$ of Eq. (\ref{eq:iHDo}) result in
\begin{eqnarray}\label{eq:inertial3form}
(\partial_t+L_u)(\verb"U"\wedge \verb"d"\verb"U") =-\verb"d"[(P-\bm{u}^2/2)\wedge \verb"d"\verb"U"],
\end{eqnarray}
whose exterior derivative leads to
\begin{eqnarray}\label{eq:inertial4form}
(\partial_t+L_u)(\verb"d"\verb"U"\wedge \verb"d"\verb"U") =0;
\end{eqnarray}
and, such a procedure can be repeated $m$ times, each of which results in
\begin{eqnarray}
(\partial_t+L_u)\verb"H"^k =-\verb"d"[(P-\bm{u}^2/2)\wedge \verb"d"(\verb"H"^{k-1})]\ \text{with}\ \verb"H"^k :=\verb"U"\wedge (\verb"d"\verb"U")^k\\\label{eq:inertial2kp1form}
\text{and}\ (\partial_t+L_u)\verb"d"(\verb"H"^{k}) =0 \ \text{with}\ \verb"d"(\verb"H"^{k})=(\verb"d"\verb"U")^{k+1}, \text{for $0\le k \le m-1$}.\label{eq:inertial2kp2form}
\end{eqnarray}

Thus, we have the `$(2k+1)$th order Kelvin circulation invariance Theorem' (the classical Kelvin theorem is then of `first order'): for each $(2k+1)$-chain $c^{2k+1}_t=\Phi_t(c^{2k+1}_0)$ transported by the flow $\Phi_t$ generated by $u$,
\begin{eqnarray}\label{eq:kthInertialKelvin}
\int_{c_t^{2k+1}}\verb"H"^k(t) = \int_{c^{2k+1}_0}\verb"H"^k(0)\ \text{for}\ \verb"H"^k(0)=\Phi_t^*[\verb"H"^k(t)]\ \text{(the `pull back')},\ \text{and}\ \label{eq:kthInertialKelvin}\\
\int_{c_t^{2k+1}}\verb"H"^k = \int_{\partial^{-1}c_t^{2k+1}}\verb"d"(\verb"H"^k) \ \text{for $c^{2k+1}_t$ being the boundary (denoted by `$\partial$') of some surface $\partial^{-1}c^{2k+1}_t$}.\label{eq:kthInertialKelvinHelmholtz}
\end{eqnarray}

Note that all the above formulations apply equally to the compressible barotropic flows \cite{ArnoldKhesin98Book}; and, the above `higher-order Kelvin circulation theorem' can also be seen in Fecko\cite{FeckoAPS2013} for the stationary - Poincar\'e - case with also an Hamiltonian mechanics example [and the nonstationary - Cartan - case in his formulation with Eq. (\ref{eq:succinct}) is just such a matter of replacing symbols. A first reflexion on Taylor's geometrical derivation of TPT reproduced in the introductory discussions might lead to the question whether the results are consistent for different orders.

\subsection{Flows in rotating frames}
Let's note three points: 
\begin{itemize}
\item First, even for compressible flows, the continuity equations for the density and entropy
are unchanged when put in the rotating frame, so we only check the momentum equation. 
\item Second, for the velocity $\bm{u}'$ relative to the \textit{R}otation,
\begin{equation}
\text{$\nabla \cdot \bm{u} = \nabla \cdot \bm{u}'$, and $\partial_z \bm{u} = \partial_z \bm{u}'$, which assures that the TPT theorem formulas read the same for $\bm{u}$ and $\bm{u}'$;}
\end{equation}
so we will drop the ``$'$'' (and, later on, also for other variables relative to any background \textit{R}eference field). 
\item Third, we should be careful in extending our intuition about rotation in $\mathbb{E}^3$ to general $\mathbb{E}^d$: For example, we can think about rotation either in a (horizontal) plane or around a (vertical) axis in $\mathbb{E}^3$, but in terms of the representation of the rotation group $SO(d)$ [and the relevant spectral theory (see below)] or the differential 2-form, the former is more fundamental, because rotation(s) may have either no or multiple axises perpendicular to the rotating plane(s). In particular, in $\mathbb{E}^4$ with only rotation in one plane (\textit{simple} rotation), we have two `rotation axises', which should not be confused with the rotation in the $\mathbb{E}^3$ part of the Minkovski space $\mathbb{M}^{1,3}$ [which, though, in some sense may be treated as $\mathbb{E}^4$ with pure imaginary time or space co-ordinate(s)] as discussed in Fecko \cite{FeckoAPS2013}.
\end{itemize}

For flows in a rotating frame, an \textit{inertial force} should thus be included in Eqs. (\ref{eq:iHDu},\ref{eq:iHDo}). This amounts to calculate the acceleration, i.e., the change rate of the velocity of a point of the rotating frame,  
which is essentially a coordinate transformation problem as presented in Chandrasekhar\cite{Chandrasekhar61Book} for $\mathbb{E}^3$, extendable of course to $\mathbb{E}^d$. Alternatively, one can adopt the functional action (Lagrangian\cite{Lynden-BellKatzRSPA81,GjajaHolmPhD96} or Hamiltonian\cite{ShashikanthJMP12}) and variational approach. The latter has been given by Shashikanth\cite{ShashikanthJMP12}, and we here further develop the results for our purpose:

$\Omega_R$ denoting the 2-form corresponding to the rotation (characterized also by an anti-symmetric matrix, the exponential of which is the rotation matrix, representing possibly multiple planar rotations: see below) 
of the frame and $\verb"X"$ the 1-form corresponding to the position vector $\bm{x}$ , the inertial force entering the left hand side of Eq. (\ref{eq:iHDu}) reads (with $\star$ denoting the Hodge dual)
\begin{equation}\label{eq:iF}
  L_{u} \verb"U"_R = (\verb"d" \iota_u + \iota_u \verb"d")[\star(\star\Omega_R \wedge \verb"X")],
\end{equation}
with the frame rotating velocity 1-form
\begin{equation}\label{eq:rotatingVelocity}
  \verb"U"_R :=\star(\star\Omega_R \wedge \verb"X").
\end{equation}
Following the terminology in $\mathbb{E}^3$, we have 
\begin{proposition}
In the frame with uniform rotation $\Omega_R$ in $\mathbb{E}^d$, the $\verb"d" \iota_u$ part on the right hand side of Eq. (\ref{eq:iF}) is exact and contributes to the \textit{centrifugal force}, and the $\iota_u \verb"d"$ part to the `\textit{Coriolis force}'
.
\end{proposition}
Later, we will further apply the \textit{spectral theorem} to the skew-symmetric matrix corresponding to $\Omega_R$, for simplifying the explicit calculations and analyses.

Combining Eqs. (\ref{eq:iHDu},\ref{eq:iHDo},\ref{eq:iF},\ref{eq:rotatingVelocity}), we have the Lie-invariance law in the rotating $\mathbb{E}^d$
\begin{eqnarray}
  (\partial_t + L_{u}) (\verb"U" + \verb"U"_R)  = -\verb"d"(P-\bm{u}^2/2),\label{eq:rLie0}\\
  (\partial_t + L_{u}) \verb"d"(\verb"U" + \verb"U"_R)  = 0.\label{eq:rLie1}
\end{eqnarray}
Thus, with the replacement of $\verb"d" \verb"U" \to \verb"d"(\verb"U" + \verb"U"_R)$, many geometrical results in the inertial frame carry over, \textit{mutatis mutandis}, in particular, the \textit{integral surface} generalizing the Helmholtz vortex line and tube, and, the relative integral invariant generalizing the circulation in Kelvin's theorem.

Actually, according to the Lie theory\cite{Fecko06Book}, rotation $R \in SO(d)$ can be represented as the exponential of an anti-symmetric matrix $A$ and we have the \textit{spectral theory}\cite{YoulaCanJMath61}:
$A$ is a normal matrix, subjecting to the \textit{spectral theorem}, thus (block) diagonalizable by a special orthogonal transformation. More definitely, for $d=2m$ we have 
$A = Q\Lambda Q^\textsf{T}$ where $Q$ is orthogonal and
\begin{eqnarray}
\text{$\Lambda =
\begin{bmatrix}
 \Lambda_1 &  & & \\
  & \Lambda_2 & & O \\
  &  & \ddots &  &\\
& O & &\Lambda_n & & & \\
    & & & & & o \\
\end{bmatrix}
, \ \ \ \ \ 
\begin{matrix}
\text{
block diagonal with $O$ and $o$ representing (blocks) of $0$s and with real \textit{rotating rate},}\\
\text{$\lambda_i$, in the $2 \times 2$ block $
\Lambda_i = \left(
    \begin{array}{cc}
    0 & \lambda_i \\
    -\lambda_i & 0 \\
    \end{array}
\right)
\ \text{for $i=1$, ..., $n$}
$,}\\
\text{being the coefficients of the purely imaginary eigenvalues: When $d=2m+1$, $\Lambda$}\\ 
\text{presents at least one extra row and column of $0$s, indicating the $d$th (rotating) axis.
}
\end{matrix}
$}
\label{eq:ST}
\end{eqnarray}

\begin{definition}
If there is only one component, say, $\lambda_i$, of $\lambda$ is nonvanishing, the rotation is \textit{simple}; otherwise, \textit{multiple}.\footnote{When we treat the rotation, represented in our special case by $\Omega_R$ but exhibiting multiple plannar rotations, as a whole, it is singular; otherwise, plural.}
\end{definition}
The spectral theorem indicates that the rotation can be decomposed into sub-rotations in mutually orthogonal 2-planes. Thus, we have
\begin{lemma}\label{lm:Om}
$\Omega_R$ can be decomposed into `orthogonal'\footnote{Note that we adopt such a terminology by the property that the wedge product of one with another never vanishes, which superficially in words seems contradicting but actually is consistent with the notion of orthogonal vectors whose inner/dot product vanishes.} 2-forms, explicitly in terms of the coordinates of $\mathbb{E}^d$, with appropriate choice of the indexes/ordering,
\begin{equation}\label{eq:Om}
  \Omega_R = \sum_{i=1}^m\lambda_i \verb"d"x_{2i-1} \wedge \verb"d"x_{2i}\ \text{for} \ 2m=d \ \text{or} \ 2m+1=d;
\end{equation}
\end{lemma}
and then follows
\begin{lemma}\label{lm:rotatingRate}
\begin{equation}
\text{For}\ \verb"U"_R = \star(\star\Omega_R \wedge \verb"X") \ \text{with a 2-form rotating rate $\Omega_R$},
\ \verb"d"\verb"U"_R = 2\Omega_R. \label{eqLemma1:rotatingRate}
\end{equation}
\end{lemma}
Thus, we can transform Eqs. (\ref{eq:rLie0},\ref{eq:rLie1}), already given by Shashikanth\cite{ShashikanthJMP12} (and traditionally, in $\mathbb{E}^3$, in terms of vector fields\cite{Chandrasekhar61Book}), into the much more explicit and convenient form for our purposes
\begin{eqnarray}
  (\partial_t + L_{u}) \verb"U" + 2\iota_u \Omega_R  = -\verb"d"(P-\bm{u}^2/2-\iota_u \verb"U"_R),\label{eq:rLie3}\\
  (\partial_t + L_{u}) (\verb"d"\verb"U" + 2\Omega_R)  = 0.\label{eq:rLie2}
\end{eqnarray}
These, especially the `Coriolis force' $2\iota_u \Omega_R$ in Eq. (\ref{eq:rLie3}), are what we are familiar with in $\mathbb{E}^3$, and now they turn out to be the same for general $\mathbb{E}^d$.

\section{Taylor-Proudman theorems and relatives in $\mathbb{E}^d$}\label{sec:TPtt}

\subsection{\label{sec:theorem}The theorem}

From Eq. (\ref{eq:rLie1}), and by the definition of Lie derivative, to first order accuracy for small $t$, the integral over a circuit (1-boundary)
\begin{equation}\label{eq:circulationLie}
  \int_{c(t)}(\verb"U"+\verb"U"_R) = \int_{c(0)}(\verb"U"+\verb"U"_R) + t\int_{c(0)} L_u (\verb"U"+\verb"U"_R) ,
\end{equation}
where $c(t)=\Phi_t[c(0)]$ and $\verb"U"+\verb"U"_R$ in the integrands on the right hand side is the pullback, $\Phi_t^*$ of the flow $\Phi_t$ generated by $u$, of that on the left hand side: $d\Phi_t^* (\verb"U"+\verb"U"_R)/dt|_{t=0} = L_u (\verb"U"+\verb"U"_R)$. And, the following \textit{integral invariant}, as the generalization of the circulation in Kelvin's theorem\cite{FeckoJGP17}, is implied by Eqs. (\ref{eq:rotatingVelocity},\ref{eq:rLie1}) in $\mathbb{E}^d$
\begin{equation}\label{eq:circulationKelvin}
  \int_{c(t)}(\verb"U"+\verb"U"_R) = \int_{c(0)}(\verb"U"+\verb"U"_R).
\end{equation}
\begin{definition}\label{def:fast}
The rotation is `fast' with some component(s) of $\lambda$ being large to make the relevant term(s) dominant. In fast rotation, each component $\lambda_i$ is effectively either large or vanishing, in the asymptotic sense, symbolically $\Omega_R \to \infty$. If more than one components of $\lambda$ are large, they can be of the same or different orders (in the sense of asymptotic analysis\cite{BenderOrszag1999Book}): in the former case, the multiple rotations are \textit{simultaneously fast}, and the latter \textit{independently fast}.
\end{definition}
\begin{remark}
If there is some nonvanishing $\lambda_i$ which is not large, with the corresponding term becoming subdominant, it will be treated as effectively vanishing in our analysis with the idea of dominant balance; if both $\lambda_i$ and $\lambda_j$ are large but $\lambda_i << \lambda_j$, then they should be treated separated (with the relevant terms set to zero subsequently: see below) as of different orders in the asymptotic analysis.
\end{remark}

For fast rotation, leaving only the terms indexed with `$_R$', Eq. (\ref{eq:Om}), Lemma \ref{lm:rotatingRate} and the Stokes theorem lead to
\begin{equation}\label{eq:circulationTaylor}
\int_{c(0)} \verb"U"_R = \int_{\partial^{-1}c(0)} \verb"d" \verb"U"_R = 2\int_{\partial^{-1}c(0)} \Omega_R = 2\sum_{i=1}^m \lambda_i \int_{\partial^{-1}c(0)} \verb"d"x_{2i-1} \wedge \verb"d"x_{2i}=: \sum_{i=1}^m 2\lambda_i A_i.
\end{equation}
That is, we have generalized Taylor's\cite{Taylor} observation:
\begin{lemma}
\label{lm:dominantCirculation}
The rotation contribution to Kelvin's circulation is the sum of the projection area $A_i$ of the circuit in the $i$th rotating plane multiplied by twice the corresponding rotation rate $\lambda_i$.
\end{lemma}

Using Eq. (\ref{eq:circulationTaylor}), we have 
\begin{equation}\label{eq:integralAreaRate}
 \int_{c(0)} L_u \verb"U"_R = \int_{\partial^{-1}c(0)} \verb"d"L_u \verb"U"_R=  2\int_{\partial^{-1}c(0)}L_u \Omega_R =0.
\end{equation}
Because for every 1-boundary $c(0)$ the second term of the right hand side should vanish for all possible surface $\partial^{-1}c(0)$ [whose boundary is $c(0)$], we have obtained (using $L_u = \iota_u \verb"d" + \verb"d" \iota_u$) the (generalized) Taylor-Proudman theorem:
\begin{theorem}\label{th:differentialAreaRate}
For a rotating flow in $\mathbb{E}^d$, with $d$ odd ($=2m+1$) or even ($=2m$), when the time variation is negligible and the rotation is fast, i.e., $\Omega_R \to \infty$, we have asymptotically
\begin{equation}\label{eq:differentialAreaRate}
L_u \Omega_R \to 0, \ \text{i.e.},\ \verb"d" \iota_u \Omega_R = \verb"d" \iota_u \sum_{i=1}^m\lambda_i \verb"d"x_{2i-1} \wedge \verb"d"x_{2i} \to 0 .
\end{equation}
\end{theorem}
\begin{proof}
See derivations above. 
\end{proof}
\begin{remark}
We have followed Taylor to take the detour to be able to see the geometry and topology, but a `textbook' shortcut would be directly taking with the dominant balance argument in Eq. (\ref{eq:rLie2}) $L_u  \Omega_R = 0$. And, bruteforce calculations without realizing nor using the above Lemma \ref{lm:rotatingRate} for definite $d$ (as we did for good checks of the examples in the next section) can also arrive at the result, but formidable for large $d$ with more than one nonvanishing plane-rotating rates ($\lambda$s). We have resisted expanding the right hand side of Eq. (\ref{eq:differentialAreaRate}) further and then analyze the formal results with Definition \ref{def:fast}, which will lead to a list of corollaries of more definite formulas for specific situations, because it is more illuminating to summarize and explain the following three constituting structures from observing the effects of the subsequent interior product and exterior derivative operators on $\Omega_R$ decomposed from Lemma \ref{lm:Om} (which makes the manipulations so simple as to border on the trivial\footnote{Readers who are not yet familiar with the manipulations of differential forms may first jump to go through the easy and concret examples in Secs. \ref{sec:E3}, \ref{sec:E4}, \ref{sec:E5} below to obtain some experience.}): 
\begin{itemize}
\item First of all, for each nonvanishing (large) $\lambda_i$, the operation of `$\iota_u$' leads to two components $u_{2i-1}\verb"d"x_{2i}$ and $-u_{2i}\verb"d"x_{2i-1}$, on which the further operation of `$\verb"d"$' results in $\lambda_i$ multiplied by $u_{2i-1,2i-1} + u_{2i,2i}$ ($\to 0$, i.e., `incompressibility in the plane' of the flow in that fast rotating plane), with the shorthand notation $u_{i,j}:=\partial_{x_j} u_i$, for the coefficient of the particular component $\verb"d"x_{2i-1} \wedge \verb"d"x_{2i}$ of the final total 2-form. `Physically', it means that actually the projection area in each (fast) rotating plane tend to be invariant, though Lemma \ref{lm:dominantCirculation} says the weighted sum. This is correct, because, for any $c$, we can take circuit $c_i$ in the $x_{2i-1}$-$x_{2i}$ plane circling exactly the projection of $c$ on this plane, thus invariant $A_i$. 
\item The operation of $\verb"d"$ after $\iota_u$ also leads to other `cross terms' with $i \ne j$, such as $(\lambda_i u_{2i-1, 2j-1} + \lambda_j u_{2j, 2i}) \verb"d"x_{2j-1}\wedge\verb"d"x_{2i}$ and $(\lambda_i u_{2i-1, 2j} - \lambda_j u_{2j-1, 2i}) \verb"d"x_{2j}\wedge\verb"d"x_{2i}$. For $r_{ij}:=\lambda_i/\lambda_j$ of order-one amplitude, i.e., \textit{simultaneously fast multiple rotations}, the dominant-balance asymptotic result $(r_{ij} u_{2i-1, 2j-1} + u_{2j, 2i}) =0=(r_{ij} u_{2i-1, 2j} - u_{2j-1, 2i})$. But, in the \textit{independently-fast} case, with $|r_{ij}| \to \infty$, further asymptotic simplification with subsequent dominant balances at different orders is that, separately, $u_{2i-1, 2j-1} = u_{2j, 2i} = u_{2i-1, 2j} = u_{2j-1, 2i} = 0$; that is, finer structures can emerge with realistic value of $|r_{ij}|>>1$.
\item If some $\lambda_i$ vanishes and/or if there is an `extra' coordinate, chosen to be $d = 2m + 1$, the cylinder conditions along the corresponding coordinate axis(es) result due to the fact that only one (large) component of $\lambda_j$ is involved in the resulted coeffecient of the 2-form component in the term containing the differential of this coordinate variable; that is, $u_{j,2i-1} = u_{j,2i} = 0$ and/or $u_{j,2m+1} = 0$, $\forall$ $j\ne i$; and, there is no TPT constraint on $u_{2i-1}$, $u_{2i}$ and/or $u_{2m+1}$, the differential of the corresponding coordinate variable(s) not appearing in our decomposition of $\Omega_R$ thus not following the interior product $\iota_u$.                                             
\end{itemize}
The above asymptotic properties correspond to some reduced models (see below) which are expected to reasonably decoupled from the whole system. Problems such as how much these sub-dynamics are self-autonomous for realistic parameters and whether the formal asymptotic states can indeed be arrived at when the parameters are taken to the limits, among others, are subtle, especially when coupled with the problem of turbulence\cite{CCEHjfm2005,EmbidMajda98,MajdaEmbid98,CambonRubinsteinGodeferdNJP04,
SagautCambon08Book,PouquetMininniPTRSA10}.
\end{remark}

Note that for $d>3$, we can have non-vanishing differential forms of order $k>3$, and it has been shown in Sec. \ref{sec:geometricalRemarks} how to extend the classical Kelvin-Helmholtz theorems to objects related to such $k$ and $d$. Thus following the above analysis, carrying $\verb"H"^{k}$ over to the rotating frame to become $\verb"H"^{k}_T := (\verb"U"+\verb"U"_R)\wedge [\verb"d"(\verb"U"+\verb"U"_R)]^k$, we have for the fast rotation $\text{with}\ \verb"H"^{k-1}_R:= \verb"U"_R \wedge (\verb"d"\verb"U"_R)^{k-1}\ \text{
dominating}\ \verb"H"^{k-1}_T$, 
\begin{corollary}
\begin{equation}
\text{When $\Omega_R \to \infty$}, \ L_u \verb"d"(\verb"H"_R^{k-1}) = 2L_u (\Omega_R)^k = 2\verb"d"\iota_u (\Omega_R)^k = 0, \ \text{for each $k$ ($=1,\ 2,\ ...,\ m-1$)}. \label{eq:kthLie}
\end{equation}
\end{corollary}
\begin{remark}\label{rm:differentialAreaRate}
Since the Lie derivative satisfies the Leibniz rule, Eq. (\ref{eq:kthLie}) can also be deduced directly from Eq. (\ref{eq:differentialAreaRate}).
\end{remark}

\subsection{\label{sc:examples}TPT Examples in $\mathbb{E}^{3,4,5}$
}
We now demonstrate in the following examples the above results which `resonate' with the inertial-wave discussions in Sec. \ref{sec:inertialWaves} and with the reduced models of Sec. \ref{sec:reducedModels}.
\subsubsection{\label{sec:E3}$\mathbb{E}^3$ case}

For $d=3$, $n=1$ in Eq. (\ref{eq:ST}) and that the rotation rate/vorticity 2-form reads
\begin{equation}\label{eq:E3rr}
  \Omega_R = \lambda \verb"d"x_1\wedge \verb"d"x_2,
\end{equation}
\textit{i.e.}, rotation in the $1$-$2$ plane or around the $x_3$ axis. The right hand side of Theorem \ref{th:differentialAreaRate},
\begin{equation}\label{eq:0coriolis}
   \lambda [(u_{1,1}+u_{2,2})\verb"d"x_1\wedge \verb"d"x_2 + u_{1,3}\verb"d"x_3\wedge \verb"d"x_2 + u_{2,3}\verb"d"x_1\wedge \verb"d"x_3] = 0,
\end{equation}
leads to the classical Taylor-Proudman relations, Eqs. (\ref{eq:cTPE},\ref{eq:cTPEuz}) which also follow Remark \ref{rm:differentialAreaRate}. If we impose incompressibility for the total flow, $u_{1,1}+u_{2,2}+u_{3,3} = 0$, then Eqs. (\ref{eq:cTPEuz}) should be replaced with $u_{3,3} = 0$. Note that the content in the square bracket of Eq. (\ref{eq:0coriolis}) is, in Taylor's notation with slight adjustments, that $[-u_{1,3}\bar{\bm{x}}_1 - u_{2,3}\bar{\bm{x}}_2+(u_{1,1}+u_{2,2})\bar{\bm{x}}_3] \cdot d\bm{S} = -[u_{1,3}\bar{\bm{x}}_1 + u_{2,3}\bar{\bm{x}}_2 - (u_{1,1}+u_{2,2})\bar{\bm{x}}_3] \cdot \bar{\bm{n}}dS \ \forall \ d\bm{S}$, where $\bar{\bm{x}}_{\bullet}$ and $\bar{\bm{n}}$ are unit vectors normal to respectively each directed coordinate plane and the surface element.

\subsubsection{\label{sec:E4}$\mathbb{E}^4$ case}
Now $\Omega_R = \lambda_1 \verb"d"x_1 \wedge \verb"d"x_2 + \lambda_2 \verb"d"x_3 \wedge \verb"d"x_4$ and Theorem \ref{th:differentialAreaRate} reads
\begin{eqnarray}
\lambda_1 [(u_{1,1}+u_{2,2})\verb"d"x_1\wedge \verb"d"x_2 + u_{1,3}\verb"d"x_3\wedge \verb"d"x_2 + u_{1,4}\verb"d"x_4\wedge \verb"d"x_2 
 - u_{2,3}\verb"d"x_3\wedge \verb"d"x_1 - u_{2,4}\verb"d"x_4\wedge \verb"d"x_1]+\nonumber\\
 +\lambda_2 [u_{3,1}\verb"d"x_1\wedge \verb"d"x_4 + u_{3,2}\verb"d"x_2\wedge \verb"d"x_4 + (u_{3,3} +u_{4,4})\verb"d"x_3\wedge \verb"d"x_4 
 - u_{4,1}\verb"d"x_1\wedge \verb"d"x_3 - u_{4,2}\verb"d"x_2\wedge \verb"d"x_3] =0 \label{eq:E4Coriolis},
\end{eqnarray}
following which (or Remark \ref{rm:differentialAreaRate}) is the 4-space Taylor-Proudman results:

In the case of constant \textit{multiple (double) rotations} for $\lambda_1 = r \lambda_2 $ and $r>0$ without loss of generality,
\begin{eqnarray}
u_{1,1} + u_{2,2} = u_{3,3} + u_{4,4} = 0 = r
u_{1,3} + u_{4,2} = r
u_{2,3} - u_{4,1} = r
u_{1,4} - u_{3,2} = r
u_{2,4} + u_{3,1},\label{eq:TP4dr1}
\end{eqnarray}
in which, coordinate transformations $x'_1 = \sqrt{r} x_1$ and $x'_2 = \sqrt{r} x_2$ can be applied to have $\lambda_1'=\lambda_2$, or, in other words, $r =1$ can be used in principle for any realistic finite $r$: For $r$ of $O(1)$ (order one for \textit{simultaneously fast} rotations), no more asymptotic behavior than Eq. (\ref{eq:TP4dr1}) can be derived; for large $r$ ($\gg 1$ for \textit{independently fast} rotations), finer structures 
\begin{equation}\label{eq:E4fiiner}
0 = u_{1,3} = u_{4,2} = 
u_{2,3} = u_{4,1} = 
u_{1,4} = u_{3,2} =
u_{2,4} = u_{3,1}
\end{equation}
are revealed by such coordinate transformation itself, or alternatively as the result of subsequent dominant balances pointed out in Remark \ref{rm:differentialAreaRate}. 

In the case of \textit{simple rotation}, for $\lambda_2 = 0$ without loss of generality,
\begin{eqnarray}
u_{1,1} + u_{2,2} = 0 = u_{1,3} = u_{1,4} = u_{2,3} = u_{2,4}.\label{eq:TPsr}
\end{eqnarray}
\textit{Cylinder condition} of the `horizontal flow' along the axises `perpendicular' to the rotating `horizontal' plane is indeed reached. There is no TPT constraint on $u_3$ nor $u_4$. 

\subsubsection{\label{sec:E5}$\mathbb{E}^5$ case}
The block-diagonal $\Lambda$ with Eq. (\ref{eq:ST}) for $d=5$ is the same as that for $d=4$ except for an additional row and column of zeros, thus no TPT constraint on $u_5$. And, the decomposition of the rotation is formally the same and we find the following extra terms in addition to those in $\mathbb{E}^4$
$$\lambda_1(u_{1,5}\verb"d"x_5\wedge \verb"d"x_2 - u_{2,5}\verb"d"x_5\wedge \verb"d"x_1) + \lambda_2(u_{3,5}\verb"d"x_5\wedge \verb"d"x_4 - u_{4,5}\verb"d"x_5\wedge \verb"d"x_3),$$
and that the Taylor-Proudman theorem in $\mathbb{E}^5$ is just adding to that in $\mathbb{E}^4$ the following cylinder condition along the rotation axis (taken to be $x_5$ here):
\begin{eqnarray}
u_{1,5} = u_{2,5} = u_{3,5} = u_{4,5} = 0\label{eq:TP5d}
\end{eqnarray}
for the fast double-rotation case. If the flow is incompressible, then the the total incompressibliity is decomposed into the first two sub-incompressibilities in Eq. (\ref{eq:TP4dr1}) plus $u_{5,5} = 0$.

For the simple rotation case with $\lambda_2=0$, there is no TPT constraint on $u_{3}$, $u_{4}$ and $u_5$.

It is then clear enough what the general TPTs would be for general $\mathbb{E}^d$, with the above results of $d=3$, $4$ and $5$. In particular, it is noted that there is always at least one coordinate axis left to be the rotating axis for odd $d$, while for even $d$, all coordinates may be involved in forming the rotating planes.

\subsubsection{\label{sec:subsummary}No simple cylinder condition(s), thus the traditional passive scalar(s), from TPT of $d>3$}
We see that it is impossible to obtain 3D passive scalars from mere fast ratations of higher-dimensional flows, because no such cylinder condition results from this kind of physical mechanism. The above are our fundamental results, which are the basis of reduced dynamical models and whose consistency with inertial wave properties will be discussed below (together with remarks on the resonanat wave theory). 

\subsection{TPT and Inertial waves}\label{sec:inertialWaves}
%

Here, we present the inertial waves, the basis for the \textit{resonant wave theory}, in $\mathbb{E}^{4}$ and $\mathbb{E}^5$, to show how the results `resonate' with TPT. For simplicity, let's now focus on the incompressible flows. 

Let's start with Eqs. (\ref{eq:iHDu},\ref{eq:iF}) for incompressible flows. In the appropriately chosen coordinates according to the spectral theorem, (\ref{eqLemma1:rotatingRate}), and $L_u = \iota_u \verb"d" + \verb"d" \iota_u$, we have
\begin{equation}\label{eq:E4RWT}
  \partial_t \verb"U" + \iota_u (\verb"d"\verb"U" + 2 \Omega_R) = -\verb"d"\{ P-\iota_u [\star(\star\Omega_R \wedge \verb"X")] \} =: -\verb"d" \Pi.
\end{equation}
The standard Fourier normal-mode analysis with $u_i = \hat{u}_i \exp\{\hat{i}(\bm{k}\cdot\bm{x} + \varpi t)\} (+ c.c.)$ and $\hat{i}^2=-1$, and similarly for $\Pi$, leads to the dispersion relation between the wave frequency $\varpi$ and vector $\bm{k}$.

The normal-mode analysis of Eq. (\ref{eq:E4RWT}) in $\mathbb{E}^4$ results in
\begin{equation}\label{eq:E4dispersion}
\left[ \begin {array}{ccccc} i\varpi &-2\,\lambda_{{1}}&0&0&ik_{{1}}
\\ \noalign{\medskip}2\,\lambda_{{1}}&i\varpi &0&0&ik_{{2}}
\\ \noalign{\medskip}0&0&i\varpi &-2\,\lambda_{{2}}&ik_{{3}}
\\ \noalign{\medskip}0&0&2\,\lambda_{{2}}&i\varpi &ik_{{4}}
\\ \noalign{\medskip}k_{{1}}&k_{{2}}&k_{{3}}&k_{{4}}&0\end {array}
 \right]
      \left[ \begin {array}{c} \hat{u}_1
\\ \noalign{\medskip}\hat{u}_2
\\ \noalign{\medskip}\hat{u}_3
\\ \noalign{\medskip}\hat{u}_4
\\ \noalign{\medskip}\hat{\Pi}\end {array}
 \right]  = 0,
\end{equation}
the existence of whose nontrivial solutions leads to the equation for the vanishing of the determinant of the coefficient matrix. The matrix being 5$\times$5 though, the determinant is actually third order in $\varpi$ due to the particular structure, and the solutions are
\begin{eqnarray}
    \text{$\varpi_{\pm} = \pm 2\sqrt{ (\lambda_2^2k_1^2 + \lambda_2^2k_2^2) + (\lambda_1^2k_3^2 + \lambda_1^2k_4^2) }/|\bm{k}|$},\label{eq:E4dispersionRelation1}\\
\text{and $\varpi_0 = 0$ }.\label{eq:E4dispersionRelation2}
\end{eqnarray}
\begin{definition}
The motion with $\varpi = 0$ corresponds to a \textit{vortical mode} which may be \textit{natural/unconditional} with no dependence on $\bm{k}$, as Eq. (\ref{eq:E4dispersionRelation2}), and which may also be \textit{imposed/conditional} with particular choice of the dependence of $\bm{k}$, as for Eq. (\ref{eq:E4dispersionRelation1}).
\end{definition}
\begin{remark}\label{rm:E4dispersion}
With $\lambda_2=0=k_4$, Eq. (\ref{eq:E4dispersionRelation1}) reduces to the well-known dispersion relations for the circularly polarized inertial waves in $\mathbb{E}^3$. Taking Eq. (\ref{eq:E4dispersionRelation2}) for the `natural' \textit{vortical mode} into Eq. (\ref{eq:E4dispersion}), we of course obtain Eq. (\ref{eq:TP4dr1}) for the TPT in $\mathbb{E}^4$. In the simple-rotation case, with $\lambda_2 = 0$, say, we see that both Eq. (\ref{eq:E4dispersionRelation2}) for the `natural' vortical/slow mode and the `imposed' vortical mode with $k_3=k_4=0$ in Eq. (\ref{eq:E4dispersionRelation1}) are consistent with the TPT (\ref{eq:TPsr}). Note that taking the above into Eq. (\ref{eq:E4dispersion}), we see that the third and fourth equations are actually $0=0$ without restriction on $\hat{u}_3$ and $\hat{u}_4$, precisely Eq. (\ref{eq:TPsr}) instead of stronger condition as one might think on the first sight. 
\end{remark}
The situation in $\mathbb{E}^5$ can be similarly discussed. 
There are four (two pairs of) dispersion relations, and, as in $\mathbb{E}^3$ or odd dimension $d=2m+1$ in general, there is no `natural' vortical mode like Eq. ({\ref{eq:E4dispersionRelation2}) in $\mathbb{E}^4$. Since the derivation is standard and the computed dispersion relations are lengthy and similar, we only present one of the four results for reference:
\begin{eqnarray}
&&\varpi = \sqrt {2} \Big{[} \big{(}\uline{\lambda_{{2}}^{4}k_{{1}}^{4}+\lambda_{{2}}^{4}k_{{2}}^{4}+
\lambda_{{2}}^{4}k_{{5}}^{4}}+
\lambda_{{1}}^{4}k_{{3}}^{4}+\lambda_{{1}}^{4}k_{{4}}^{4}+\lambda_{{1}}^{4}k_{{5}}^{4}+ \nonumber \\
&&\uline{ 2\lambda_{{2}}^{4}k_{{2}}^{2}k_{{5}}^{2}+2\lambda_{{2}}^{4}k_{{2}}^{2}k_{{1}}^{2}
+2\lambda_{{2}}^{4}k_{{5}}^{2}k_{{1}}^{2} } + 2\lambda_{{1}}^{4}k_{{4}}^{2}k_{{3}}^{2}+
2\lambda_{{1}}^{4}k_{{4}}^{2}k_{{5}}^{2}+2\lambda_{{1}}^{4}k_{{3}}^{2}k_{{5}}^{2}\nonumber \\
&&+ \uline{ 2\lambda_{{2}}^{2}k_{{2}}^{2}\lambda_{{1}}^{2}k_{{4}}^{2}+
2\lambda_{{2}}^{2}k_{{2}}^{2}\lambda_{{1}}^{2}k_{{3}}^{2}+2\lambda_{{1}}^{2}k_{{4}}^{2}
\lambda_{{2}}^{2}k_{{1}}^{2}+2\lambda_{{1}}^{2}k_{{3}}^{2}
\lambda_{{2}}^{2}k_{{1}}^{2} } - \nonumber \\
&& \uwave{ 2k_{{5}}^{2}\lambda_{{1}}^{2}
\lambda_{{2}}^{2}k_{{1}}^{2}-2k_{{5}}^{2}\lambda_{{1}}^{2}\lambda_{{2}}^{2}k_{{2}}^{2}-2k_{{5}}^{2}\lambda_{{1}}^{2}\lambda_{{2}}^{2}k_{{3}}^{2}-2k_{{5}}^{2}\lambda_{{1}}^{2}\lambda_{{2}}^{2}k_{{4}}^{2}-2k_{{5}}^{4}\lambda_{{1}}^{2}\lambda_{{2}}^{2} } \big{)}^{\frac{1}{2}} \nonumber \\
&&+ \uline{ \lambda_{{2}}^{2}k_{{1}}^{2} + \lambda_{{2}}^{2}k_{{2}}^{2}  +\lambda_{{2}}^{2}k_{{5}}^{2} } + \lambda_{{1}}^{2}k_{{3}}^{2}+\lambda_{{1}}^{2}k_{{4}}^{2}+\lambda_{{1}}^{2}k_{{5}}^{2} 
\Big{]}^{\frac{1}{2}} 
\Big{/} 
\big{(} k_{{5}}^{2}+k_{{4}}^{2}+k_{{3}}^{2}+k_{{2}}^{2}+k_{{1}}^{2} \big{)}^{\frac{1}{2}}.\label{eq:E5dispersionRelation}
\end{eqnarray}
Looking lengthy though, the above relation actually presents nice and reasonable order (as can be observed with the help of the underwave and underlines to highlight and separate the different groups of terms), and interested readers can check (it is easiest to start with the simple rotation with $\lambda_2=0$), as an example, how the above mode reduces to the vortical one and recovers the TPT with additional (\ref{eq:TP5d}).

Such properties indeed present the connections with TPT, satisfying our main purpose for introducing these relatives, thus appear to be the indication of a  `resonant wave theory' as in $\mathbb{E}^3$ for describing part of the scenario (waves may not be complete for a turbuence\footnote{In fact, to our best knowledge, so far there is no argument for  precise realization of TPT state: Waleffe\cite{WaleffePoF93} argued with resonant-interaction theory the energy transfer to slower but not vortical modes, thus only quasi-two-dimensionalization, not precisely two-dimensionalization as the TPT state.}) of approaching those Taylor-Proudman states: A simple consistent intuition is that inertial waves are anisotropic, while resonant interactions due to nonlinearity tend to isotropize the system, thus reducing the former. We also believe that other higher-order geometrical laws discussed in Sec. \ref{sec:geometricalRemarks} should work in constraining the higher-order (resonant) interactions to favor the dynamical evolution towards the Taylor-Proudman-type states, because in the Lie-carrying equations for higher-order differential forms, say, $\verb"d"\verb"U"\wedge \verb"d"\verb"U"$, higer-order nonlinearities are involved, which naturally connects the higher-order resonant interactions. Although there are rigorous mathematical proofs/estimations\cite{EmbidMajda98,MajdaEmbid98} in $\mathbb{E}^3$ supporting the resonant wave theory under appropriate conditions, there are also unsettled subtleties\cite{SagautCambon08Book,CCEHjfm2005,SmithLeeJFM2005} in the turbulent regime involving the large Reynolds number, near resonant interactions and higher order effects in the long time\cite{NewellJFM1969}. In $\mathbb{E}^4$, the number and distribution of the resonant and near-resonant modes are at least quantitatively different to $\mathbb{E}^3$, thus may offer a model with changing effects of the output to check relevant ideas developed to attack the subtle issues in the latter. We speculate that with the change of dimensionality and that the geometry and topology of the slow manifold, or the change of the population and distribution of the (near) resonant modes\cite{SmithLeeJFM2005,ClarkDiLeoniMininniJFM2016}, there could be some `phase transition' point, the number of dimensionality, crossing which the asymptotic limits of fast rotation are different. Note that existing simulations for $d>3$, such as those for isotropic cases\cite{GotohPRE2007,MiyazakiPhD2010}, may be extended to check some of our results; adding the noninertial force to that of channel flow\cite{NikitinJFM2011} may also be illuminating for rotating channel turbulence.

\subsection{TPT and Reduced models}\label{sec:reducedModels}

As said, one of our motivations is also to check whether rotation(s) in high dimensions also  lead to (simplified) models of `realistic' relevance. If this is the case, then we can unify different problems dynamically, with the possibilities of new angles of view, physical insights and techniques. The direct question is whether we can obtain the passive-(multiple-)scalar model from fast rotations, just as the case of two-dimensional passive scalar resulting from imposing TPT on the 3D Navier-Stokes flow.

In the derivation of the TPTs, the time variation of the (relative) flow is much smaller than the rotating rate, or simply let it be zero (steady state), as the fast rotating limit. Intuitively, with increasing rotation rate, the period of time for developing fluid structures varying as fast as the rotation should also increase, as the excitation of ever smaller scales is indicated. Thus, at least for some finite time interval, the TPT features should dominate the dynamics.
In the formal derivation of the averaged equation using the standard asymptotic expansion involving the multiple-scale method, the reduced dynamics are seen to be the average over the fast waves or over the vertical co-ordinate (according to the dispersion relation in $\mathbb{E}^3$ mentioned in Remark \ref{rm:E4dispersion}.) The result for the incompressible flow is well documented\cite{BabinMahalovNikolaenkoEJMB1996,EmbidMajdaCPDE1996} with the averaged (sub)-system being `autonomous' and working as the `catalyst' for the fast/wave modes, which is consistent with taking the (stationary) TPT $\partial_3 \bm{u} = 0$ back to the dynamics. Such a time-dependent form of TPT character is, as said, conceptually simple, but technically rather subtle, especially when the issue of turbulence is involved and if high level of mathematical rigor is to be pursued\cite{CambonRubinsteinGodeferdNJP04,PouquetMininniPTRSA10,GibbonHolmNonlinearity2017}. The well-known \textit{resonant wave theory} describes part of the scenario, but not the formation of pure two-dimensional structure\cite{Greenspan1968Book,WaleffePoF93}. 

\subsubsection{$\mathbb{E}^3$ and beyond: reduction of compressible flows}\label{sec:reducedCompressible}
If the flow is compressible, although TPT still requires $\partial_z \bm{u}_h = 0$ and $\nabla_h \cdot \bm{u}_h = 0$, the compressive-mode existence implies that simplely averaging $u_3$ over the inertial wave or the vertical axis do not make much sense. 
Indeed, consider the TPT in $\mathbb{E}^3$ 
for the compressible but barotropic fluid, for simplicity [with 
$\Pi_{\bullet}:=\int d P_{\bullet}/\rho$ below], rotating fast (compared to the vortical part of the relative motion) in $\mathbb{E}^3$. Formally using Eq. (\ref{eq:cTPE}) and keeping the time derivative, with the equivalence of the indexes `$_v$', `$_3$' and `$_z$', the 3D Navier-Stokes equation reduces to
\begin{eqnarray}
\partial_t \rho + \nabla\cdot (\rho\bm{u})=0,\label{eq:continuity}\\
\!\!\!\! (\partial_t + \bm{u}_h\cdot \nabla_h - \nu_h \Delta_h) \bm{u}_h  = 
-\nabla_h \Pi_h
,\label{eq:uh3dsr}\\
(\partial_t + u_z\partial_z + \bm{u}_h\cdot\nabla_h - \nu \Delta_h - \nu \partial_{zz}) u_z =
-\partial_z \Pi_z
.\label{eq:uv3dsr}
\end{eqnarray}
\begin{remark}\label{rm:E3passive}
Eq. (\ref{eq:uh3dsr}) can also be understood to be from the vertical average of the original Navier-Stokes equation, thus the corresponding meanings of $\bm{u}_h$ and $\Pi_h$. According to the TPT, we may also impose $\nabla_h \cdot \bm{u}_h = 0$ on this equation. $\Pi_z$ in principle is just $\Pi$ whose vertical average is $\Pi_h$. The $\bm{u}_h$ dynamics appears to be coupled with others through $\Pi_h$ and the continuity equation (\ref{eq:continuity}), but it could be autonomous with $\Pi_h$, if not affected externally, classically determined by itself through a Poisson equation, as is the common wisdom for a pure incompressible flow problem (whose well-posedness in the sense of global regularity however has not been proved). It is possible that $\Pi_h$ and $\bm{u}_h$ simply affect $u_z$ and $\rho$ in a one-way manner, but it is not very clear at this point when and why $u_z$ and $\rho$ would not affect $\Pi_h$ and $\bm{u}_h$. If $u_z$ in Eq. (\ref{eq:uv3dsr}) indeed has no back-reaction onto $\bm{u}_h$, we may call $u_z$ a \textit{nonlinear passive scalar}.  One might expect, with an extra source for $u_z$, this source could in general affect the $\bm{u}_h$ dynamics through $P_z$, which deserves further examination: For example, one may check, with $u_z$ numerically forced in rotating flows, whether the dependence of $\bm{u}_h$ on this forcing is decreased with increasing rotating rates, by comparing with an independent simulation of Eq. (\ref{eq:uh3dsr}).

Note that in the incompressible case, thus vanishing $\partial_z u_z$ and $u_z\partial_z u_z$ in the above, Eq. (\ref{eq:uv3dsr}) can also be viewed as the vertical average of that of the original Navier-Stokes equations. 
Our extra pressure term now should vanish, if the average is over a torus or over an infinite length. Otherwise, with $g(x,y):=\partial_z \Pi_z$, we have as a representation of $\Pi$, $\Pi_z = zg(x,y,t) + f(x, y, t)$. We then see that $\nabla_h g = \bm{0}$ (otherwise we would have $z$-depedent $\Pi_h$), thus $\nabla_h \Pi_z = \bm{0}$ (no contribution/feedback onto the horizontal flow through $\nabla_h \Pi_h$) and $-\partial_z \Pi_z = g(t)$ uniform in space. For such a linear passive scalar problem, it means only a trivial linear surperposition of a function to that without such a term in the classical resonant-wave-theory result, introducing no essential difference. \textit{The incompressible flows in $\mathbb{E}^{2m+1}$ with $m$ \textit{simultaneously fast} rotations obviously all formally share such $(2m)$D-$(2m+1)$C characteristics, but interesting novel feature may emerge (c.f., the two extended remarks for $\mathbb{E}^5$ in Sec. \ref{sec:E4reductions} below).} 
\end{remark}

\subsubsection{$\mathbb{E}^4$ and beyond: reduction of incompressible flows}\label{sec:E4reductions} 
The ambiguity of passiveness in Remark \ref{rm:E3passive} appears to be due to the compressibility. We then turn to the TPT of incompressibe flow 
in $\mathbb{E}^4$. 

\paragraph{$\mathbb{E}^4$ simple rotation}\label{sec:E4simple} 
Denoting the incompressible horizontal velocity in the $x_1$-$x_2$ $h$-plane by $\bm{u}_h = [u_1(x_1,x_2,t), u_2(x_1,x_2,t)]$, the incompressible vertical velocity `along' the $v$-plane $x_3$-$x_4$ by $\bm{u}_v = [u_3(x_1,x_2,x_3,x_4,t),u_4(x_1,x_2,x_3,x_4,t)]$ with homogeneous unit density $\rho=1$
, we use the TPT Eq. (\ref{eq:TPsr}) for \textit{simple} fast rotation to formally decompose the incompressible 4D Navier-Stokes into
\begin{eqnarray}
\!\!\!\!\!\! \partial_t \bm{u}_h + (\bm{u}_h\cdot \nabla_h - \nu_h \Delta_h) \bm{u}_h  =  - \nabla_h P_h\ \text{with} \ \nabla_h \cdot \bm{u}_h =0, \ \nabla_v \bm{u}_h=\bm{0}, \ \nabla_v {P}_h=\bm{0},\label{eq:uh4dsr}\\
\!\!\!\!\!\! \text{and}\ \partial_t \bm{u}_v + (\bm{u}_v\cdot \nabla_v + \bm{u}_h\cdot\nabla_h - \nu_h \Delta_h - \nu_v \Delta_v) \bm{u}_v =  -\nabla_v P_v \ \text{with} \ \nabla_v \cdot \bm{u}_v=0, \label{eq:uv4dsr}
\end{eqnarray}
where $\nu_h$ and $\nu_v$ may be different to take into account the possibility of anisotropic damping for whatever reason (including the re-scaling with a non-unit $r$). 
\begin{remark}\label{rm:E4passive}
Now Eq. (\ref{eq:uh4dsr}) can again be viewed as the average over the fast waves or the vertical axises ($x_3$ and $x_4$), thus the corresponding meanings of $\bm{u}_h$ and $P_h$, and the dynamics could be `self-autonomous' as in the classical incompressible flow problem. However, again, $P_v$ represents $P$ whose vertical average (over both $x_3$ and $x_4$) is $P_h$, and it is neither very clear at this point when and why $P_v=P$ would not affect $P_h$. We expect again, at least, if $\bm{u}_v$ is stirred by an extra force, this force should in general affect the $\bm{u}_h$ dynamics through $P_v$. Obviously, \textit{this is the general feature for any $d>3$ with less than $m$ (\textit{simultaneously}) \textit{fast} rotations leaving more than one rotation axises}.
\end{remark}

\paragraph{$\mathbb{E}^4$ double rotations}\label{sec:E4double} 
We now consider incompressible flow in $\mathbb{E}^4$ with \textit{double rotations}, respectively \textit{simultaneously fast} and \textit{independently fast}:--- 

\subparagraph{$\mathbb{E}^4$ \textit{simultaneously fast} double rotations}\label{sec:E4simultaneousDouble} 
Eq. (\ref{eq:TP4dr1}), for $r=1$, say, does not appears to indicate any simple reduced model. An extended remark for $\mathbb{E}^5$ with such rotations, concerning the $u_{5,5} = 0$ complementing Eq. (\ref{eq:TP5d}), is that, $u_5$ now is a passive scalar advected by the four-dimensional `horizontal' flow in the two rotating planes.

\subparagraph{$\mathbb{E}^4$ \textit{independently fast} double rotations}\label{sec:E4independentDouble} 
Eqs. (\ref{eq:TP4dr1}, \ref{eq:E4fiiner}) indicate two sets of Eq. (\ref{eq:uh4dsr}) for $\bm{u}_h$ denoting  $(u_1,u_2)$ and  $(u_3,u_4)$ respectively, which can be understood to be for the dynamics averaged over $x_3, x_4$ and $x_1, x_2$ respectively. In other words, we have Eq. (\ref{eq:uh4dsr}) with the replacement of the index notation $h \leftrightarrow m$ and $v \leftrightarrow 3-m$ for $m=1$ and $2$, and, separation of the pressure $P=P_1(x_1,x_2,t) + P_2(x_3,x_4,t)$, resulting in two formally identical independent self-autonomous incompressible flows, neither of which is passively advected by the other. Note that, now, we can not introduce extra forcing to couple the two flows, otherwise the force introduced would first break the flow's independence on the vertical axises. An extended remark for $\mathbb{E}^5$ with such rotations, concerning the $u_{5,5} = 0$ complementing Eq. (\ref{eq:TP5d}), is that, $u_5$ now is a passive scalar advected by the two independent flows respectively in the two rotating planes.

The definite discussions on $d=4$ may also be appropriately extended to even higher dimensions, but, together with the above discussions, the indications of both simplicity and complexity for connecting TPT and the reduced models will be so clear that we don't need to go any further for the purpose of this note. 

\subsection{TPT and higher-order Kelvin theorems: consistency and higher-order correction}

We have seen that the rotation dominant contribution from higher-order (generalized) Kelvin theorem can be deduced from the first-order theorem, but it makes sense to check what exactly higher-order Kelvin theorems say for the rotation dominant state and what would be the sub-dominant (in terms of $\verb"U"_R$) correction from the higher-order Kelvin theorem. We can of course also write down the general results corresponding to any $k$th-order Kelvin theorem, which, however, is not as illuminating as the illustrative approach for presentation with examples below.

It is enough to check the third-order object $\verb"H"^1_T =(\verb"U"+\verb"U"_R)\wedge \verb"d"(\verb"U"+\verb"U"_R)$, whose integral invariant as the `circulation' over a `circuit' (any 3-boundary) in the rotation dominance approximation, $\verb"d"L_u \verb"U"_R \wedge \verb"d"\verb"U"_R = 0$ or $\verb"d"\iota_u \Omega_R \wedge \Omega_R = 0$ in Eq. (\ref{eq:kthLie}), leads to $\sum_{i = 1}^4  u_{i,i} = 0$ in $\mathbb{E}^4$ with double rotation and to null constraint from fast simple rotation, and, additionally $\sum_{i = 1}^4  u_{1,i} = 0$ in $\mathbb{E}^5$ with double rotation and null constraint again for simple rotation, nothing new compared to what we have already derived in the last (sub)sections (as may be expected from the fact that the higher-order geometrical objects are simply derived from the lower-order ones). And, the next-order (in $\verb"U"_R$) contribution
\begin{equation}
\verb"d"\verb"U"_R \wedge \verb"d"\verb"U" - \verb"d"\verb"U" \wedge \verb"d"\verb"U"_R = 2\verb"d"\verb"U"_R \wedge \verb"d"\verb"U",\ \text{in}\ \verb"d"\verb"H"_T^1,\ \text{to}\ \verb"d" L_u \verb"H"_T^1 =0\ \text{is null},
\end{equation}
in the sense that the correction with $L_u \verb"d"\verb"U" \wedge \verb"d"\verb"U"_R = 0$ would be in terms of quadratic forms (second-order combinations) of the derivatives of $\bm{u}$, from $L_u$ and $\verb"d"\verb"U"$. That is, sub-dominant corrections of $\verb"U"_R$ turn into the higher-order corrections of $\bm{u}$, which do not modify the Taylor-Proudman theorems that are of first-order-in-$u$ nature. 

The above concret demonstration is generic and can be summarized as the following
\begin{corollary}\label{crl:highOrderKelvin}
The dominant results from higher-order Kelvin theorems are consistent with the TPT from conventional (first-order) Kelvin theorem, and sub-dominant (in $\Omega_R$) contributions reveal higher-order (in $u$) corrections, irrelevant to the leading-order TPT.
\end{corollary}

\section{Further discussions}\label{sec:FurtherDiscussions}
The formal result of TPT for (compressible) flows in $\mathbb{E}^3$ (Sec. \ref{sec:E3}) and the time-dependent consideration of reduced models (Sec. \ref{sec:reducedCompressible}) appear to be of most direct realistic relevance. The relevant issues of the traditional incompressible case have been discussed over around one and a half centuries with the origin traced back to the 19th century\cite{VelascoFuentesEJMB2009}, but the compressible ones have attracted much less attention. We have proven that, unlike the formal reduction of a two-dimensional passive scalar from fast-rotating 3D incompressible flow, it is not possible to obtain 3D passive scalar(s) by mere rotation(s) in high spatial dimensions, which does not satisfy one of our motivations. The byproduct that other nonlinear passive scalar dynamics may result is however still interesting. The extension of the analyses to plasma fluids is direct, with nontrivial technicalities though, and will be communicated elsewhere together with other discussions. 

We particularly remark that due to the reduction of horizontal/perpendicular compressibility in time dependent flows, we may anticipate that aeroacoustic noise, of particular mechanical engineering interests, and particle acceleration, of astronomical relevance, may accordingly be reduced anisotropically; but, even for a strong gas turbulence (isotropic or not), full of rotatory structures, with or without background rotation, it is also possible to use the result to develope a theory and technique for reducing the compressibility (thus the relevant noise and particle heating, among others) of flows, i.e., fastening a turbulent gas with the circuits and/or the accordingly bounded surfaces used first by Taylor to prove TPT.

On the mathematical side, high-dimensional flows have been extensively analyzed\cite{DongDuCMP2006,DongGuDPDE2014,GuJDE2014,LiuWangJDE2017}, including four and six dimensional (magneto)hydrodynamics, and it is expected that similar analyses with our Coriolis forces shall bring new lights on the regularity issue, as those in 3D\cite{BabinMahalovNicolaenkoAsymptotAnal1997,KoniecznyYonedaJDE2011,SunYangCuiAMPA2016,
ZhaoWangMMAS,KishimotoYonedaMathAnn2017,KozonoMashikoTakadaJEE2014}, with mutual enlightenment. Resonant-wave relevant theory concerning the reduced models discussed in Sec. \ref{sec:reducedModels} also calls for rigorous analysis for more concret knowledge.

As remarked in the introductory discussions (Footnote \ref{lb:dDflow}), there is a long list of literatures on turbulence in high spatial dimensions. Though the theoretical physics concern appears to be on the `gross' dimensional effects, in particular the critical dimension concerning intermittency/mean-field theory emengence or cascade-direction transition. We believe that turbulence of hight spatial dimensions, as a sea of various simple- and multiple-rotation structures of various parameters, the detailed properties of intensive structures with fast rotation(s) studied here may play important roles in understanding some fundamental issues.

\section*{Acknowledgments}
This work is supported by NSFC (No. 11672102) and the Ti\'an-Yu\'an-Xu\'e-P\`ai research foundation.

\providecommand{\url}[1]{\texttt{#1}}
\providecommand{\urlprefix}{}
\providecommand{\foreignlanguage}[2]{#2}
\providecommand{\Capitalize}[1]{\uppercase{#1}}
\providecommand{\capitalize}[1]{\expandafter\Capitalize#1}
\providecommand{\bibliographycite}[1]{\cite{#1}}
\providecommand{\bbland}{and}
\providecommand{\bblchap}{chap.}
\providecommand{\bblchapter}{chapter}
\providecommand{\bbletal}{et~al.}
\providecommand{\bbleditors}{editors}
\providecommand{\bbleds}{eds: }
\providecommand{\bbleditor}{editor}
\providecommand{\bbled}{ed.}
\providecommand{\bbledition}{edition}
\providecommand{\bbledn}{ed.}
\providecommand{\bbleidp}{page}
\providecommand{\bbleidpp}{pages}
\providecommand{\bblerratum}{erratum}
\providecommand{\bblin}{in}
\providecommand{\bblmthesis}{Master's thesis}
\providecommand{\bblno}{no.}
\providecommand{\bblnumber}{number}
\providecommand{\bblof}{of}
\providecommand{\bblpage}{page}
\providecommand{\bblpages}{pages}
\providecommand{\bblp}{p}
\providecommand{\bblphdthesis}{Ph.D. thesis}
\providecommand{\bblpp}{pp}
\providecommand{\bbltechrep}{}
\providecommand{\bbltechreport}{Technical Report}
\providecommand{\bblvolume}{volume}
\providecommand{\bblvol}{Vol.}
\providecommand{\bbljan}{January}
\providecommand{\bblfeb}{February}
\providecommand{\bblmar}{March}
\providecommand{\bblapr}{April}
\providecommand{\bblmay}{May}
\providecommand{\bbljun}{June}
\providecommand{\bbljul}{July}
\providecommand{\bblaug}{August}
\providecommand{\bblsep}{September}
\providecommand{\bbloct}{October}
\providecommand{\bblnov}{November}
\providecommand{\bbldec}{December}
\providecommand{\bblfirst}{First}
\providecommand{\bblfirsto}{1st}
\providecommand{\bblsecond}{Second}
\providecommand{\bblsecondo}{2nd}
\providecommand{\bblthird}{Third}
\providecommand{\bblthirdo}{3rd}
\providecommand{\bblfourth}{Fourth}
\providecommand{\bblfourtho}{4th}
\providecommand{\bblfifth}{Fifth}
\providecommand{\bblfiftho}{5th}
\providecommand{\bblst}{st}
\providecommand{\bblnd}{nd}
\providecommand{\bblrd}{rd}
\providecommand{\bblth}{th}

\end{document}